\documentclass[12pt,leqno,twoside]{amsart}
\usepackage{amssymb,amsmath,amsthm,soul,color}
\usepackage{t1enc}
\usepackage{a4,indentfirst,latexsym}
\usepackage{graphics}
\usepackage{mathrsfs}
\usepackage{cite,enumitem,graphicx}
\usepackage[colorlinks=true,urlcolor=blue,
citecolor=red,linkcolor=blue,linktocpage,pdfpagelabels,
bookmarksnumbered,bookmarksopen]{hyperref}
\usepackage[english]{babel}
\usepackage[left=2.61cm,right=2.61cm,top=2.72cm,bottom=2.72cm]{geometry}
\usepackage[metapost]{mfpic}
\usepackage[normalem]{ulem}

\usepackage[hyperpageref]{backref} 

\newtheorem{Th}{Theorem}[section]

\newtheorem{Lem}[Th]{Lemma}
\newtheorem{Cor}[Th]{Corollary}

\newenvironment{altproof}[1]
{\noindent
{\em Proof of {#1}}.}
{\nopagebreak\mbox{}\hfill $\Box$\par\addvspace{0.5cm}}

\newcommand{\wt}{\widetilde}

   \newcommand{\vp}{\varphi}
   
   \newcommand{\eps}{\varepsilon}

   \def\div{\mathop{\mathrm{div}}}

     \def\span{\mathrm{span}}
   
   \def\supp{\mathrm{supp}}

   \def\d{\diamond}

   \def\R{\mathbb{R}}

   \def\curl{\mathrm{curl}}

   \def\V{\mathcal{V}}

   \def\Wc{\mathscr{W}}
   \def\W{\mathcal{W}}
   
   \def\C{\mathbb{C}}

    \newcommand{\UU}{u}

    \newcommand{\SO}{\cS\cO}

\def\rh{\rightharpoonup}

\def\rn{\mathbb{R}^N}

\newcommand{\cC}{{\mathcal C}}
\newcommand{\cD}{{\mathcal D}}
\newcommand{\cDp}{\cD^{1,p}(\R^3,\R^3)}

\newcommand{\cH}{{\mathcal H}}

\newcommand{\cM}{{\mathcal M}}
\newcommand{\cN}{{\mathcal N}}
\newcommand{\cO}{{\mathcal O}}

\newcommand{\cS}{{\mathcal S}}
\newcommand{\cT}{{\mathcal T}}

\newcommand{\cV}{{\mathcal V}}
\newcommand{\cW}{{\mathcal W}}

\newcommand{\Om}{\Omega}

\def\r{\mathbb{R}}
\def\r3{\mathbb{R}^3}
\def\Wc{\cD^{1,p}(\curl;\R^3)}
\def\Wco{\cD^{1,p}_{\cO}(\curl;\R^3)}
\def\rn{\mathbb{R}^N}

\def\eps{\varepsilon}
\def\rh{\rightharpoonup}

\def\ir3{\int_{\r3}}
\def\vp{\varphi}
\def\wt{\widetilde}

\def\ol{\overline}

\def\lam{\lambda}
\def\lam1{\lambda_1}

\def\d1p{\mathcal{D}^{1,p}}
\def\cm{\mathcal{M}}

\def\curlop{\nabla\times}
\newcommand{\weakto}{\rightharpoonup}

\numberwithin{equation}{section}

\title{Optimizers in Sobolev-curl inequalities}

\author[J. Mederski]{Jaros\l aw Mederski$^{*}$}
\author[A. Szulkin]{Andrzej Szulkin}

\address[J. Mederski]{\newline\indent
	{Institute of Mathematics,
		\newline\indent 
		Polish Academy of Sciences,
		\newline\indent 
		ul. \'Sniadeckich 8, 00-656 Warsaw, Poland,		\newline\indent
		and
		\newline\indent 
		Faculty of Mathematics and Computer Science,		\newline\indent 
		Nicolas Copernicus University, \newline\indent ul. Chopina 12/18,		 87-100 Toruñ, Poland
	}
}
\email{\href{mailto:jmederski@impan.pl}{jmederski@impan.pl}}

\address[A. Szulkin]{\newline\indent 
	Department of Mathematics,
	\newline\indent 
	Stockholm University,
	\newline\indent 
	106 91 Stockholm, Sweden
}
\email{\href{mailto:andrzejs@math.su.se}{andrzejs@math.su.se}}


\thanks{$^{*}$Corresponding author.}

\begin{document}

\begin{abstract} 
We study a Sobolev-type inequality involving the $p$-curl operator in $\mathbb{R}^3$. We prove the existence of a minimizer $u:\mathbb{R}^3\to \mathbb{R}^3$ which yields a solution to the $p$-curl-curl equation in the critical case
$$
\curlop (|\curlop u|^{p-2}\curlop u)= |u|^{p^*-2}u\quad \hbox{in }\mathbb{R}^3.$$	
The cases $p=2$ and $p=3/2$ are motivated, respectively, by nonlinear
Maxwell equations and zero modes of three-dimensional Dirac operators. The infinite-dimensional kernel of curl and the critical exponent
prevent a direct application of standard compactness arguments.
Our proof combines local compactness under the nonlinear constraint
$\operatorname{div}(|u|^{p^*-2}u)=0$ and
a new variational approach that allows to treat quasilinear strongly indefinite problems by direct minimization on a Nehari-type constraint. 
We also establish existence and compactness results in axially
symmetric classes, including two types of solutions for $p=3/2$
distinct from the explicit Loss--Yau fields. Finally, the same
minimization strategy gives a new proof of compactness modulo
translations and dilations for minimizing sequences
in the classical critical Sobolev inequality.
\end{abstract}

\maketitle

\section{Introduction} \label{sec:intro}

Sobolev inequalities and embedding theorems play a crucial role in the theory of partial differential equations as well as in mathematical physics. 
We recall the pioneering result by Aubin \cite{Aubin} and Talenti  \cite{Talenti}.
Suppose that $N\geq 2$, $1<p<N$ and let $S_p=S_p(N)$ be the Sobolev constant, i.e. the largest constant such that the inequality
\begin{equation}\label{eq:Sobolevineq}
	\int_{\R^N}|\nabla u|^p\, dx\geq S_p\Big(\int_{\R^N}|u|^{p^*}\,dx\Big)^{\frac{p}{p^*}}
\end{equation}
holds for any $u\in \cD^{1,p}(\R^N)$ where $p^*:=\frac{Np}{N-p}$ and $\cD^{1,p}(\R^N)$ denotes the homogeneous Sobolev space, i.e  the completion of $\cC^{\infty}_0(\R^N)$ with respect to the norm  $|\nabla \cdot|_{p}$. Here and in the sequel $|\cdot|_q$ denotes the $L^q$-norm for $q\in[1,\infty]$.  Equality holds in \eqref{eq:Sobolevineq} if $u$ has the form
\begin{equation}\label{eq:AT}
u(x)=\big(a+b|x|^{\frac{p}{p-1}}\big)^{\frac{p-N}{p}}
\end{equation}
for some constants $a,b>0$ and the formula for $S_p$ is given in \cite{Talenti}.  Moreover, recently  Damascelli et al. \cite{SciunziAiM2014} and Sciunzi \cite{SciunziAiM2016} showed  that \eqref{eq:AT} constitute unique positive solutions to the problem
\begin{equation}\label{eq:plap}
-\div(|\nabla u|^{p-2}\nabla u)=|u|^{p^*-2}u\hbox{ in }\R^N,
\end{equation}
up to translations in $\R^N$ and under some additional relation between $a$ and $b$.

Throughout the rest of the paper except for Section \ref{sec:Lions} we assume $N=3$.  Motivated
 by the classical inequality \eqref{eq:Sobolevineq} and by physical considerations which we shall discuss later, we want to analyse an analogous inequality  with  $\nabla u$ replaced by $\curlop u$, the curl of a vector field $u:\R^3\to\R^3$. The obvious candidate $S$ for a corresponding inequality, the largest possible constant such that 
$$
\int_{\r3}|\curlop u|^p\, dx\geq S\Big(\int_{\r3}|u|^{p^*}\,dx\Big)^{\frac{3-p}{3}}
$$ 
holds for $u\in\cDp$, is certainly not suitable because it equals $0$. Indeed, $\curlop u $ has a  nontrivial (in fact an infinite-dimensional) kernel: $\curlop (\nabla\vp)=0$ for all $\vp\in\cC_0^{\infty}(\R^3)$. 

Throughout the paper $\curlop u$ as well as the divergence of $u$, $\div u$, should be understood in the distributional sense.

In order to properly define a Sobolev-type inequality we first introduce the Banach space 
$$
\Wc := \big\{u\in L^{p^*}(\R^3,\R^3): \curlop u\in L^p(\R^3,\R^3)\big\}
$$
which we endow with the norm 
$$
\|u\|:= \left(|u|^2_{p^*}+|\curlop u|^2_p\right)^{1/2}.
$$
 It is standard to show that
$\Wc$ is the closure of $\cC^{\infty}_0(\R^3,\R^3)$ with respect to $\|\cdot\|$, see Lemma \ref{density}.

Denote the kernel of $\curlop (\cdot )$ in $\Wc$  by
$$\cW:=\big\{w\in \Wc: \curlop w=0\big\},$$
and let $S_{p,\curl}$ be the largest possible constant such that the {\em Sobolev-curl inequality}
\begin{equation}\label{eq:neq}
\int_{\r3}|\curlop u|^p\, dx\geq S_{p,\curl}\inf_{w\in \cW}\Big(\int_{\r3}|u+w|^{p^*}\,dx\Big)^{\frac{3-p}{3}}
\end{equation}
holds for any $u\in \Wc\setminus \W$. Inequality \eqref{eq:neq} is in fact trivially satisfied also for $u\in \cW$ because then both sides are zero.

The inequality \eqref{eq:neq} has been introduced for the first time by  the authors in \cite{MedSz} in the case $p=2$, in the context of nonlinear Maxwell equations and the quintic effect of the medium, see also \cite{Mederski,MederskiJFA2018} and the references therein. One of the main results of \cite{MedSz} states that $S_{2,\curl}>S_2$ and equality holds in \eqref{eq:neq} for $u\in  \cD^{1,2}(\curl;\R^3)\setminus \W$ which is a ground state solution to the critical curl-curl problem 
\begin{equation} \label{eq10}
\curlop (\curlop u)= |u|^{4}u.
\end{equation}
It has also been observed that $u$ cannot be radially symmetric (i.e. $\SO(3)$-equivariant) and a question concerning symmetry properties of ground states has been left as an open problem. In \cite{GaczkowskiMederskiSchino} it has been recently proved that there are infinitely many  cylindrically symmetric solutions $(u_n)$
of the form
\begin{equation*}
u_n(x)=\frac{v_n(x)}{\sqrt{x_1^2+x_2^2}}\Bigl(\begin{smallmatrix}
-x_2\\
x_1\\
0
\end{smallmatrix}\Bigr),\quad x=(x_1,x_2,x_3)\in\R^3\setminus(\{0\}\times\{0\}\times\R),
\end{equation*}
where $v_n:\R^3\to\R$ is a cylindrically symmetric function, i.e. it is invariant with respect to the action of $\SO(2)\times\{1\}\subset \SO(3)$. We emphasize the physical relevance of analyzing \eqref{eq10}. Solutions to this problem correspond to time-harmonic fields 
$E(x,t)=u(x)\cos(\omega t)$ which solve the full nonlinear electromagnetic wave equation when combined with Maxwell's laws and material relations \cite{Agrawal,BartschMederski1,Stuart:1993}.
Due to the complexity of the exact model, various approximations are commonly used which lead to a (nonlinear) Schr\"odinger equation. They often involve neglecting the term $\nabla (\div(u))$ in the identity $\curlop (\curlop u)=\nabla (\div(u))-\Delta u$ (the so-called scalar approximation), or applying the slowly varying envelope approximation. However, such simplifications can yield {\em non-physical} results; see \cite{Akhmediev-etal, Ciattoni-etal:2005}. The case $p=3/2$, which has a very different physical background, has been considered in \cite{FrankLoss2}. See the comments below.

Our main aim in this work is to extend the analysis of \cite{FrankLoss2, MedSz} to a general $p$ satisfying $1<p<3$ and find solutions to  the {\em $p$-curl-curl problem} with critical exponent
\begin{equation}
\label{eq}
\curlop (|\curlop u|^{p-2}\curlop u)= |u|^{p^*-2}u\quad \hbox{in }\R^3.
\end{equation}
Note that in general for divergence-free $u$, $|\curlop u|_p\neq |\nabla u|_p$ unless $p=2$ which complicates some arguments as we shall see later. In the quasilinear case $p\neq 2$ the strongly indefinite structure of 
\eqref{eq}  is more difficult to treat and we have to take into account the nonlinear nature of the $p$-curl-curl operator.

We want to emphasize that the case $p=3/2$ has a special and very important significance in the theory of zero modes of the three-dimensional Dirac equations. Namely, in view of the results of Fr\"ohlich, Lieb, Loss and Yau \cite{FrolichLiebLoss,LossYau}, the existence of zero modes is related to the stability of the hydrogen atom, i.e. finiteness of its ground state energy. Recall that
 a {\em zero mode} is a  nontrivial solution $\psi:\R^3\to\C^2$ to the spinor equation
\begin{equation}\label{eq:zeromode}
\sigma \cdot (-i\nabla -u)\psi =0
\end{equation}
where $\sigma$ stands for the vector of Pauli matrices
$$\sigma_1=\begin{pmatrix}
0& 1\\
1 & 0
\end{pmatrix},\; \sigma_2=\begin{pmatrix}
0& -i\\
i & 0
\end{pmatrix},\;\sigma_3=\begin{pmatrix}
1& 0\\
0 & -1
\end{pmatrix} $$
and $u$ is a vector potential for the magnetic field $\curlop u$. Loss and Yau \cite{LossYau} showed that if
\begin{eqnarray}
u(x)&:=&3(1+|x|^2)^{-2}\big((1-|x|^2)w+2(w\cdot x)x+2w\times x)\big) \label{defu}
\end{eqnarray}
where $w\ne 0$ is a constant vector, then  \eqref{eq:zeromode} is satisfied for an explicitly given $\psi$ and properly chosen $w$, see \cite[Section II]{LossYau} for the details.

Very recently Frank and Loss \cite{FrankLoss1,FrankLoss2,FrankLoss3} have considered a necessary condition for the existence of a zero mode which they express in terms of $S_{3/2,\curl}$, and this constant plays a crucial role in their analysis. We would like to underline that in their arguments conformal invariance (which is present only if $p=3/2$) has been used in order to prove a nonlinear variant of the Rellich-Kondrachov theorem and to deal with the lack of compactness of the problem. In \cite{FrankLoss2} it has also been noted that $u$ given by \eqref{defu} satisfies the Euler-Lagrange equation associated with $S_{3/2,\curl}$. We shall show in Theorem \ref{Th:main2} that in fact this $u$ satisfies \eqref{eq} with $p=3/2$ if $|w|=4/3$.

The aim of this work is to analyze the problem in the full range $1 < p < 3$ and to provide new results also in the cases $p = 2$ and $p = 3/2$. We present a new variational approach that allows to treat quasilinear strongly indefinite problems by a direct minimization on a Nehari-type constraint as will be seen in the proof of Theorem \ref{Th:main1}.
We show that $S_{p,\curl}$ is attained and  that any optimal vector field $u$ for $S_{p,\curl}$ is (up to rescaling) a ground state solution to~\eqref{eq}. 

The same variational strategy can also be applied to different Sobolev-type inequalities, in particular to the scalar
	Sobolev inequality and yields a new proof of the classical
	concentration-compactness result of Lions \cite{Lions} stated in
	Theorem \ref{Th:LionsW} below.

Let us define the {\em energy} functional associated with \eqref{eq} by setting
\begin{equation}\label{eq:action}
	J(u):=\frac1p\int_{\R^3} |\curlop u|^p\,dx- \frac1{p^*}\int_{\R^3} |u|^{p^*}\, dx,\quad u\in \Wc,
	\end{equation}
and  introduce a Nehari-type constraint
\begin{equation}\label{def:Neh}
	\cN:=\Big\{u\in \Wc\setminus\cW: \int_{\R^3}|\curlop  u|^p=\int_{\R^3}|u|^{p^*}\, dx\hbox{ and }\div(|u|^{p^*-2}u)=0\Big\}.
	\end{equation}
We also introduce the following constant $\cH_p$ which is largest positive with the property that the inequality
\begin{equation}\label{eq:neqH_p}
	\int_{\r3}|\curlop v|^p\, dx\geq \cH_p \int_{\r3}|\nabla v|^p\, dx
\end{equation}
holds for any
$$v\in \cV:=\left\{v\in \Wc: \div(v) =0\right\}.$$
Clearly, $\cH_2=1$ since $\curlop (\curlop v)=-\Delta v$ as $\div(v)=0$, and hence $|\curlop v|_2=|\nabla v|_2$. It is not a priori clear that $\cH_p>0$. In Corollary \ref{corhp} we shall show that in fact $\cH_p\in (0,2^{p/2}]$.

Our first main result reads as follows.

\begin{Th}\label{Th:main1}
(a) $S_{p,\curl} > S_p\cdot\cH_p$. \\
(b)  If $(u_n)\subset \cN$ is a minimizing sequence for $J$, then there are $(s_n)\subset (0,\infty)$ and $(y_n)\subset \R^3$ such that, passing to a subsequence, 
$$
s_n^{3/p^*}u_n(s_n\cdot+y_n)\to u
$$ 
where $u$ is a minimizer for $J$ on $\cN$.\\
(c) $\inf_{\cN}J=\frac13S_{p,\curl}^{3/p}$ is attained,  $u$ is a ground state solution to \eqref{eq} and equality holds in \eqref{eq:neq} for this $u$. If $u\in\Wc\setminus\cW$ satisfies equality in \eqref{eq:neq}, then there are unique $t>0$ and $w\in\cW$ such that $t(u+w)\in\cN$ and $J(t(u+w))=\inf_{\cN} J$.
\end{Th}
 
We emphasize that the proof of this theorem is effected via a direct minimization argument, without making use of a Palais-Smale sequence related to $(u_n)$.

Let 
\begin{equation} \label{defM}
	\cM:=\Big\{u\in \Wc: \div(|u|^{p^*-2}u)=0\Big\}.
\end{equation}
Both $\cM$ and $\cN$ are topological manifolds, see Sections \ref{sec:setting} and \ref{sec:Om=R3}. As a consequence of Theorem \ref{Th:main1}(b) we obtain the following corollary.

\begin{Cor}\label{Cor}
If $(u_n)\subset \cM$ is a minimizing sequence for $S_{p,\curl}$ such that $|u_n|_{p^*}=1$, then there are $(s_n)\subset (0,\infty)$ and $(y_n)\subset \R^3$ such that, passing to a subsequence, 
$$\big(s_n^{3/p^*}u_n(s_n\cdot+y_n)\big)$$ is convergent to a minimizer for $S_{p,\curl}$.
\end{Cor}

Parts (a) and (c) of Theorem \ref{Th:main1} have been obtained in \cite{MedSz} for $p=2$ whereas Corollary \ref{Cor} constitutes the main result of \cite{FrankLoss2} for $p=3/2$.

A natural question arises whether ground states have some symmetry properties. It is easy to see that  any $\SO(3)$-equivariant vector field is the gradient of a radial function (see e.g. \cite[Lemma 4]{Bartsch}, \cite[Theorem 1.3]{ClappSzulkin}). Therefore any $\SO(3)$-equivariant (weak) solution to \eqref{eq} is trivial, and in particular a ground state cannot be radially symmetric. Instead we consider the subgroup $\cO:=\SO(2)\times\{1\}\subset \SO(3)$ and let $\Wco$ be the subspace of  $\Wc$ consisting of $\cO$-equivariant vector fields. Similarly as above, let $S_{p,\curl}^{\cO}$ be the largest possible constant such that the Sobolev-curl inequality \eqref{eq:neq} holds for any $u\in \Wco\setminus \W$. We obtain a symmetric variant of Theorem \ref{Th:main1}; moreover, we gain some additional insights if $p=3/2$.

\begin{Th}\label{Th:main2}
	(a) $S_{p,\curl}^{\cO}\geq S_{p,\curl}$. Moreover, $S_{3/2,\curl}^{\cO}\leq4\pi$.\\
	(b)  If $(u_n)\subset \cN_\cO:=\cN\cap\Wco$ is a minimizing sequence for $J$, then there are $(s_n)\subset (0,\infty)$ and $(y_n)\subset \{0\}\times\{0\}\times\R$ such that, passing to a subsequence, 
	$$
	s_n^{3/p^*}u_n(s_n\cdot+y_n)\to u
	$$ 
	where $u$ is a minimizer for $J$ on $\cN_\cO$.\\
	(c) $\inf_{\cN_\cO}J=\frac13\big(S_{p,\curl}^{\cO}\big)^{3/p}$ is attained and  $u$ is a solution to \eqref{eq}.
\\
(d)  $u$ of the form \eqref{defu} is a solution to \eqref{eq} with $p=3/2$ provided that $|w|=4/3$. It is $\cO$-equivariant if $w=(0,0, \pm4/3)$. 
\end{Th}
It is not clear if $u$ of the form \eqref{defu} with $|w|=4/3$ is a ground state solution  to \eqref{eq} with $p=3/2$, hence we do not know whether a strict inequality holds in the first statement of Theorem \ref{Th:main2}(a).

As we shall see below, we can find new solutions to \eqref{eq} which, even in the case $p=3/2$, are different from  $u$ of the form \eqref{defu}.

In physics and mathematics literature devoted to Maxwell's equations one considers fields of the form
\begin{equation}\label{eq:formulauU}
u(x)=\beta(x)\Bigl(\begin{smallmatrix}
-x_2\\
x_1\\
0
\end{smallmatrix}\Bigr),\quad x=(x_1,x_2,x_3)\in\R^3\setminus(\{0\}\times\{0\}\times\R),
\end{equation}
where $\beta(x)=\beta(r,x_3)$ with $r=|(x_1,x_2)|$. Note that
for such $u$, $\div(\UU)=0$ and $\curlop(\curlop u)=-\Delta u$ for $x\in\R^3\setminus(\{0\}\times\{0\}\times\R)$, and a natural question arises if $S_{p,\curl}$ is attained by $u$ as in \eqref{eq:formulauU}. In \cite{FrankLoss2} it has been conjectured that $S_{3/2,\curl}$ is attained by $u$ of the form \eqref{defu}. 
If this conjecture is true, it would be reasonable to expect that for no $p\in(1,3)$ can $u$ of the form \eqref{eq:formulauU} realize $S_{p,\curl}$. 

In a similar way as in \cite[Section 2]{AzzBenDApFor}, if $u\in\Wco$, then there exist $\cO$-equivariant  $u_{\rho}$, $u_{\tau}$, $u_{\zeta}\in\Wco$ such that for every $x\in \R^3\setminus(\{0\}\times\{0\}\times\R)$,
$\UU_{\rho}(x)$ (resp. $\UU_{\tau}(x)$, $\UU_{\zeta}(x)$) is the projection of $\UU(x)$ on $\span \{(x_1,x_2,0)\}$ (resp. $\span\{(-x_2,x_1,0)\}$, $\span\{(0,0,1)\}$). 
Hence $u\in\Wco$  is of the form \eqref{eq:formulauU} if and only if $u$ is invariant with respect to the action
$$
\cT(\UU)=\cT(\UU_\rho+\UU_\tau+\UU_\zeta):=-\UU_\rho+\UU_\tau-\UU_\zeta.
$$
Then 
\begin{eqnarray*}
	\Wco^{\cT}&:=&\big\{u\in \Wco: \cT(u)=u\big\}.
\end{eqnarray*}
In a similar way we define 
\begin{eqnarray*}
	\Wco^{\cS}&:=&\big\{u\in \Wco: \cS(u):=-\cT(u)=u\big\}
\end{eqnarray*}
which consists of vector fields of the form 
\begin{equation}\label{eq:formulauU2}
u(x)=\alpha(x)\Bigl(\begin{smallmatrix}
x_1\\
x_2\\
0
\end{smallmatrix}\Bigr)+\gamma(x)\Bigl(\begin{smallmatrix}
0\\
0\\
1
\end{smallmatrix}\Bigr),\quad x=(x_1,x_2,x_3)\in\R^3\setminus(\{0\}\times\{0\}\times\R)
\end{equation}
where  $\alpha(x)=\alpha(r,x_3)$ and $\gamma(x)=\gamma(r,x_3)$. Again,  in a similar way, let $S_{p,\curl}^{\cT}$ (resp. $S_{p,\curl}^{\cS}$) be the largest possible constant such that the Sobolev-curl inequality \eqref{eq:neq} holds for any $u\in \Wco^{\cT}\setminus \W$ (resp. $u\in \Wco^{\cS}\setminus \W$).  Let $ \cN_\cT:=\cN\cap\big(\Wco\big)^{\cT}$ and $\cN_\cS:=\cN\cap\big(\Wco\big)^{\cS}$. There holds:

\begin{Th}\label{Th:main3}
	(a) $S_{p,\curl}^{\cT}\geq S_{p,\curl}^{\cO}$ and  $S_{p,\curl}^{\cS}\geq S_{p,\curl}^{\cO}$.\\
	(b)  If $(u_n)\subset \cN_\cT$ (resp. $(u_n)\subset \cN_\cS$) is a minimizing sequence for $J$, then there are $(s_n)\subset (0,\infty)$ and $(y_n)\subset \{0\}\times\{0\}\times\R$ such that, passing to a subsequence, 
	$$
	s_n^{3/p^*}u_n(s_n\cdot+y_n)\to u
	$$ 
	where $u$ is a minimizer of $J$ on $\cN_\cT$  (resp. $\cN_\cS$).\\
	(c) $\inf_{\cN_\cT}J=\frac13\big(S_{p,\curl}^{\cT}\big)^{3/p}$ and $\inf_{\cN_\cS}J=\frac13\big(S_{p,\curl}^{\cS}\big)^{3/p}$ are attained  and the minimizers are solutions  to \eqref{eq} of the form \eqref{eq:formulauU} and  \eqref{eq:formulauU2}.
\end{Th}

Observe that $u$ given by  \eqref{defu} is not  of the form \eqref{eq:formulauU} or \eqref{eq:formulauU2}, hence in Theorem \ref{Th:main3} we have found two new different  solutions for $p=3/2$.

 Finally, in the last section of the paper we show that our approach by minimization also provides a new proof of the following classical result established by Lions \cite{Lions}. 

\begin{Th}\label{Th:LionsW}
	The Sobolev inequality \eqref{eq:Sobolevineq} has an optimizer 
	$u \in \cD^{1,p}(\R^N)$. 
	Moreover, any minimizing sequence in $\cD^{1,p}(\R^N)$ is, up to translations, dilations and multiplication by constants, 
	relatively compact in $\cD^{1,p}(\R^N)$. If $u$ is an optimizer, then there is a unique $t>0$ such that $u_0:=tu$ is a weak solution to \eqref{eq:plap}.
\end{Th}

Recall that Lions \cite{Lions} dealt with general sequences converging weakly in $\cD^{1,p}(\R^N)$ and developed a concentration-compactness principle. In the final step this principle was applied to minimizing sequences yielding compactness up to translations and dilations. 
Very recently Dietze and Nam \cite{DNam} found a simplified version of the concentration-compactness principle and applied it in order to get a shorter proof of the above result. 

In our approach we do not use the above concentration-compactness arguments. Instead we project a minimizing sequence on a Nehari-type manifold and apply a different argument based on a result by Solimini \cite{Solimini} which states that {\em a bounded sequence in $\cD^{1,p}(\R^N)$ converges to zero in $L^{p^*}(\R^N)$ if and only if every translation and dilation of it converges weakly to zero in $\cD^{1,p}(\R^N)$}. This allows us to give a new, different proof of Theorem \ref{Th:LionsW}.

The paper is organized as follows. In Section \ref{sec:setting} we introduce a functional setting and perform a general concentration-compactness analysis for this case. We show that the topological manifold $\cM$
is locally compactly embedded in $L^q(\R^3,\R^3)$ for $1\leq q<p^*$ and that if a sequence $(u_n)$ is contained in this manifold and $u_n\rh u$, then $u_n\to u$ a.e. after passing to a subsequence. This result will play an essential role  in the proof of Theorem \ref{Th:main1}. 
Note that $\cD^{1,p}(\curl;\R^3)$ is not locally compactly embedded in $L^q(\R^3,\R^3)$ for any $q$, so the condition $\div(|u|^{p^*-2}u)=0$ is crucial here.
The proof of the first main result is given in Section \ref{sec:Om=R3}.
In Section \ref{sec:symmetry} we find symmetric solutions and prove Theorems \ref{Th:main2} and \ref{Th:main3}, and in  Section \ref{sec:Lions} we give a new proof of Theorem \ref{Th:LionsW} mentioned above.

\section{Functional setting and preliminaries}\label{sec:setting}

Let $\cD^{1,p}(\R^3,\R^3)$ denote the completion of $\cC^{\infty}_0(\R^3,\R^3)$ with respect to the norm $|\nabla \cdot|_p$. In what follows $\lesssim$ and $\gtrsim$ denote the inequalities up to a multiplicative constant.
The following Helmholtz decomposition holds (see \cite{Mederski,MSchSz} for the case $p=2$).

\begin{Lem}\label{defof} $\V$ and $\cW$ are closed subspaces of $\Wc$  and
	\begin{equation}\label{HelmholzDec}
	\Wc=\V\oplus \W
	\end{equation}
where $$\cV:=\Big\{v\in \Wc: \int_{\R^3}\langle v,\nabla \varphi\rangle\,dx=0
\text{ for every $\varphi\in \cC^\infty_0(\R^3)$} \Big\}$$
and
$$
\cW:=\Big\{w\in \Wc: \int_{\R^3}\langle w,\curlop \varphi\rangle\,dx=0
\text{ for every $\varphi\in \cC^\infty_0(\R^3,\R^3)$} \Big\}.
$$
Moreover, $\cV\subset\cD^{1,p}(\r3,\r3)$ and the norms $|\nabla \cdot|_p$ and $|\curlop \cdot|_p$ are equivalent in $\V$. 
\end{Lem}

Note that in the sense of distributions $\cV = \{v\in\Wc: \div(v)=0\}$ and $\cW = \{w\in\Wc: \curlop w=0\}$

\begin{proof} 
We follow similar arguments as in \cite[Lemma 2.4]{MSchSz} for $p=2$; however, since in general for divergence-free $u$, $|\curlop u|_p\neq |\nabla u|_p$ if $p\neq 2$, we have to modify the proof.
 
 Firstly, we easily check that $\V$ and $\cW$ are closed subspaces of $\Wc$.
	 
Now, take any $u\in \Wc$ and $u_n\in\cC_0^{\infty}(\R^3,\R^3)$ such that $u_n\to u$ in $\Wc$. Let $\vp_n\in\cC^{\infty}(\R^3)$ be the Newtonian potential of $\div(u_n)$, i.e. $\vp_n$ solves $\Delta \vp_n = \div(u_n)$.  Since $u_n\in\cC_0^{\infty}(\R^3,\R^3)$, then by \cite[Proposition 1]{Iwaniec}, $\nabla \vp_n\in L^r(\R^3,\R^3)$ for every $r\in(1,\infty)$. In particular, $\nabla\vp_n\in L^{p^*}(\R^3,\R^3)$. 
Hence
$$	v_n:=u_n-\nabla \vp_n\in  L^{p^*}(\R^3,\R^3).
	$$ 
	Note also that  $\curlop v_n = \curlop u_n$  and $\div(v_n)=0$ pointwise. Let $v_n=(v_n^1,v_n^2,v_n^3)$ and let $e_i$ denote the $i$-th element of the standard basis in $\r3$ (so $v_n^i = \langle v_n,e_i\rangle$). Since
$$-\Delta  v_n=\curlop \curlop v_n,$$
employing the vector calculus identity
\[
\div(A\times B) = \langle \curlop A, B\rangle - \langle A,\curlop B\rangle
\] 
with $A = \curlop v_n$ and $B=e_i$, we obtain
\[
-\Delta v_n^i = \langle \curlop(\curlop v_n), e_i\rangle = \div((\curlop v_n)\times e_i) =: \div(f_n^i), \quad i=1,2,3. 
\]
Note that $\curlop v_n = \curlop u_n\in L^2(\R^3,\R^3)\cap L^p(\R^3,\R^3)$ and
$f_n^i\in L^2(\R^3,\R^3)\cap L^p(\R^3,\R^3)$. Again  by \cite[Proposition 1]{Iwaniec}, 
	$$|\nabla v_n^i|_p\leq C_p|f_n^i|_p \quad\hbox{for }i=1,2,3$$
where $C_p>0$ is a constant.
Hence 
\begin{equation}\label{eqes3C}
|\nabla v_n|_p\leq D_p |\curlop v_n|_p = D_p |\curlop u_n|_p,
\end{equation}
for a constant $D_p>0$ depending on $C_p$.
Similarly we obtain that for $m,n\ge 1$,
	$$|\nabla(v_n-v_m)|_{p}\leq D_p|\curlop(v_n-v_m)|_{p}=D_p|\curlop(u_n-u_m)|_{p}\leq D_p \|u_n-u_m\|.$$
	Thus $(v_n)$ is a Cauchy sequence in $\cDp$. Let $v:=\lim_{n\to\infty}v_n$ in $\cDp$. Then
	$$\int_{\R^3}\langle v,\nabla \vp\rangle\,dx=\lim_{n\to\infty}\int_{\R^3}\langle v_n,\nabla \vp\rangle\,dx=0$$
	for any $\vp\in\cC_0^{\infty}(\R^3)$, hence $\div(v) =0$ and $v\in\cV$.
	Moreover, as $|\curlop v|^2 \le 2|\nabla v|^2$, we have 
	\begin{equation} \label{ineq5}
	|\curlop(v_n-v)|_{p}\leq 2^{1/2}|\nabla(v_n-v)|_{p}\to 0,
	\end{equation}
	so $v_n\to v$ in  $\Wc$ and
	$\nabla \vp_n=u_n-v_n\to u-v$ in $\Wc$. Since $\nabla\vp_n\in\cW$ and $\cW$ is closed,
	$u-v\in \cW$ and we get the decomposition
	$$u=v+(u-v)\in \cV+\cW.$$
	\indent Now take $v\in \cV\cap\W$ and since $\curlop v =0$,
	 by \cite[Lemma 1.1(i)]{Le}, $v=\nabla\xi$ for some $\xi\in W^{1,{p^*}}_{loc}(\R^3)$. Since $\div(v) = 0$, $\xi$ is harmonic and so is $v$. Since $v\in L^{p^*}(\R^3,\R^3)$, by the mean-value formula we infer that $v=0$, so \eqref{HelmholzDec} holds.

From \eqref{eqes3C} with $v$ replacing $v_n$ and \eqref{ineq5} with $v$ replacing $v_n-v$ we see that
\begin{equation} \label{doubleineq}
2^{-1/2}|\curlop v|_p \le  |\nabla v|_p\leq D_p |\curlop v|_p,
\end{equation}
i.e. the norms $|\nabla\cdot|_p$ and $|\curlop\cdot|_p$ are equivalent in $\cV$.
\end{proof}

\begin{Cor} \label{corhp}
$\cH_p\in(0,2^{p/2}]$ where $\cH_p$ is the constant introduced in \eqref{eq:neqH_p}.
\end{Cor}

\begin{proof}
According to \eqref{doubleineq}, $\cH_p^{1/p} \ge D_p^{-1}>0$ and
\[
|\curlop v|_p \ge \cH_p^{1/p}|\nabla v|_p \ge \cH_p^{1/p}\,2^{-1/2}|\curlop v|_p.
\]
Hence $\cH_p\in(0,2^{p/2}]$.
\end{proof}

In Theorem \ref{Th:Concentration} below we formulate a concentration-compactness-type result in $\cM$.
By convexity, for any $v\in\cV$ we find a unique $w(v)\in \cW$ such that
\begin{equation}\label{eq:ineqF}
	\int_{\R^3}|v+w(v)|^{p^*}\,dx \le \int_{\R^3}|v+ w|^{p^*}\,dx \quad \text{for all } w\in \cW.
\end{equation}
This implies that
\begin{equation} \label{eq:eqf}
	\int_{\R^3}\langle |v+w|^{p^*-2}(v+w),\zeta\rangle\,dx  =0 \quad \text{for all } \zeta\in\cW \text{ if and only if }  w=w(v).
\end{equation}
Since in particular we can take $\zeta = \nabla\vp$ with $\vp\in \cC_0^\infty(\r3)$, equivalently we have 
\begin{equation} \label{eq:eqf1}
v+w\in\cM \text{ where } v\in\cV,\ w\in\cW \quad  \text{if and only if} \ w=w(v)
\end{equation}
(recall the definition \eqref{defM} of $\cM$).

\begin{Lem} \label{convw}
The mapping $v\mapsto w(v)$ is continuous from $\cV$ to $\cW$ in the $L^{p^*}$-topology, hence also in $\Wc$. In particular, $\cM$ is a topological manifold.
\end{Lem}

\begin{proof}
Suppose $v_n,v\in \cV$ and $v_n\to v$ in the $L^{p^*}$-topology. By \eqref{eq:ineqF}, $(w(v_n))$ is bounded, hence $w(v_n)\rh w$ for some $w\in\cW$ after passing to a subsequence. So
\begin{eqnarray*}
\lim_{n\to\infty}\ir3|v_n+w(v_n)|^{p^*}\,dx & \ge & \ir3|v+w|^{p^*}\,dx \ge \ir3|v+w(v)|^{p^*}\,dx \\
& = &  \lim_{n\to\infty}\ir3|v_n+w(v)|^{p^*}\,dx \ge \lim_{n\to\infty}\ir3|v_n+w(v_n)|^{p^*}\,dx,
\end{eqnarray*}
 thus $\lim_{n\to\infty}\ir3|v_n+w(v_n)|^{p^*}\,dx = \ir3|v+w(v)|^{p^*}\,dx$ and $w(v_n)\to w(v)$ in $\cW$.
\end{proof}

Denote the space of finite measures in $\R^3$ by $\cM(\R^3)$.

\begin{Th}(cf. \cite[Theorem 3.1]{MedSz})  \label{Th:Concentration}
	Suppose $v_n+w(v_n)\in\cM$ for $n\geq 1$,  $v_n\weakto v_0$ in $\cV$, $v_n\to v_0$ a.e. in $\R^3$, $|\nabla v_n|^p\weakto \mu$ and $|v_n|^{p^*}\rh \rho$
	 in $\cM(\R^3)$ as $n\to\infty$. Then
	there exist an at most countable set $I\subset\R^3$ and nonnegative weights $\{\mu_x\}_{x\in I}$, $\{\rho_x\}_{x\in I}$ such that
	$$
	\mu\geq |\nabla v_0|^p+\sum_{x\in I}\mu_x\delta_x, \quad \rho = |v_0|^{p^*} + \sum_{x\in I}\rho_x\delta_x,
	$$
	 and passing to a subsequence, $w(v_n)\rh w(v_0)$ in $\cW$, $w(v_n)\to w(v_0)$ a.e. in $\R^3$ and in $L^q_{loc}(\R^3,\R^3)$ for any $1\leq  q<p^*$. In particular, $\cM$ is compactly embedded  in $L^q_{loc}(\R^3,\R^3)$ for any $1\leq  q<p^*$.
\end{Th}

Although the structure of the proof is the same as in \cite[Theorem 3.1]{MedSz}, for the sake of completeness and since some important adaptation is needed (in particular if $p^*<2$), we provide a detailed argument in Appendix \ref{appendix:concom}.

\section{Proof of Theorem \ref{Th:main1}} \label{sec:Om=R3}

It is clear that a minimizer $w(u)$ in \eqref{eq:ineqF} exists uniquely for any $u\in \Wc$, not only for $u\in \cV$. In view of Lemma \ref{defof}, $u+w(u)=v+w(v)\in \cV\oplus\cW$ for some $v\in\cV$ and therefore
\begin{equation} \label{eqeq}
\inf_{w\in\cW} \ir3 |u+w|^{p^*}\,dx = \ir3|u+w(u)|^{p^*}\,dx = \ir3|v+w(v)|^{p^*}\,dx
\end{equation}
and
\begin{equation} \label{eq:scurl}
S_{p,\curl} = \inf_{\substack{u\in \Wc \\ \curlop u\neq 0}} \frac{|\nabla\times u|_p^p}{|u+w(u)|_{p^*}^p} = \inf_{v\in \cV\setminus\{0\}} \frac{|\curlop v|_p^p}{|v+w(v)|_{p^*}^p}.
\end{equation}

\begin{Lem}\label{lemmaScurlS}
$S_{p,\curl}\geq S_p\cdot\cH_p$.
\end{Lem}

\begin{proof}
Given $\eps>0$, by \eqref{eq:scurl} and the definition \eqref{eq:neqH_p} of $\cH_p$ we can find $v\in\cV\setminus\{0\}$ such that
\begin{equation} \label{eq:5}
\cH_p\int_{\R^3}|\nabla v|^p\, dx\leq \int_{\R^3}|\curlop v|^p\, dx\le (S_{p, \curl}+\eps)\Big( \int_{\R^3}|v+w(v)|^{p^*}\,dx\Big)^{\frac{p}{p^*}}.
\end{equation}
Let $v=(v_1,v_2,v_3)$.  In Appendix \ref{sec:relation} we shall show that $|v| \in \cD^{1,p}(\r3)$ and 
\begin{equation} \label{CauchySchwartz}
\left|\nabla|v|\right| = \big(\sum_{i=1}^3|v\cdot\partial_iv|^2\big)^{1/2}/|v| \le |\nabla v| \quad\text{a.e. in } \R^3
\end{equation}
 (here the middle term should be understood as 0 if $v(x)=0$). Assuming this, in view of
the Sobolev inequality we have
\begin{eqnarray}\label{eq:SpcurlSpH}
\int_{\R^3}|\nabla v|^p\,dx&\geq& \int_{\R^3}\big|\nabla |v|\big|^p\,dx\geq  S_p\Big(\int_{\R^3}|v|^{p^*}\,dx\Big)^{p/p^*}\\
&\geq & S_p\Big(\int_{\R^3}|v+w(v)|^{p^*}\,dx\Big)^{p/p^*}\nonumber
\end{eqnarray}
and taking into account \eqref{eq:5},  we get
 $S_{p,\curl}+\eps\geq S_p\cdot\cH_p$ for all $\eps>0$ and the conclusion follows. 
\end{proof}

For $s>0$, $y\in\R^3$ and $u:\R^3\to\R^3$ we denote $T_{s,y}(u):= s^{3/p^*}u(s\cdot +y)$. 
The following lemma is a special case of \cite[Theorem 1]{Solimini}, see also \cite[Lemma 5.3]{Tintarev}.

\begin{Lem}\label{lem:Solimini}
Suppose that $(v_n)\subset \cDp$ is bounded. Then $v_n\to 0$ in $L^{p^*}(\R^3,\R^3)$ if and only if $T_{s_n,y_n}(v_n)\weakto 0$ in $\cDp$ for all $(s_n)\subset (0,\infty)$ and $(y_n)\subset \R^3$.
\end{Lem}

Observe that 
$T_{s,y}$ is an isometric isomorphism of $\Wc$ which leaves the functional $J$ and the subspaces $\cV, \cW$ invariant. In particular, $w(T_{s,y}u)=T_{s,y}w(u)$.

By \eqref{def:Neh} and \eqref{defM}, $$\cN = \{u\in \cM\setminus\{0\}: J'(u)u=0\}.$$ 
It can be shown that $\mathcal{N}$ is a topological manifold. We shall not use this fact explicitly here; 
for a proof we refer the reader to \cite[Section~4]{MedSz}. 
Although the case treated there is $p=2$, the argument carries over without essential modifications.

\begin{Lem} \label{equiv}
Suppose $u+w(u)\in \cN$. Then 
\[
\frac{|\nabla\times u|_p^p}{|u+w(u)|_{p^*}^p} = A \quad \text{if and only if} \quad J(u+w(u)) = \frac13 A^{3/p}.
\]
In particular, $\inf_\cN J = \frac13S_{p,\curl}^{3/p}>0$.
\end{Lem}

\begin{proof}
Since $u+w(u)\in \cN$, we have $|\nabla\times u|_p^p = |u+w(u)|_{p^*}^{p^*}$. Hence $\curlop u\ne 0$,  
\[
\frac{|\nabla\times u|_p^p}{|u+w(u)|_{p^*}^p} =  |u+w(u)|_{p^*}^{p^*-p} \quad \text{and} \quad J(u+w(u)) = \frac13|u+w(u)|_{p^*}^{p^*}.
\] 
This gives the first conclusion. The second one follows now from \eqref{eq:scurl}.
\end{proof}

\medskip

\begin{altproof}{Theorem \ref{Th:main1}} We prove part (b) first. 

Let  $(u_n)\subset \cN$ be a minimizing sequence  and write $u_n=v_n+w(v_n)\in\cV\oplus\cW$. Since
\begin{eqnarray} \label{ac}
J(u_n)&=&J(u_n) -\frac1{p^*}J'(u_n)u_n = \Big(\frac1p-\frac{1}{p^*}\Big)|\curlop u_n|^p_p = \frac13|\curlop v_n|^p_p,\\\label{ca}
J(u_n)&=&J(u_n) -\frac1pJ'(u_n)u_n  = \frac13|u_n|^{p^*}_{p^*},
\end{eqnarray}
it follows that $(u_n)$ is bounded.  By Lemma \ref{equiv}, $J(u_n)$ is bounded away from $0$, hence $|u_n|_{p^*}$ and  $|v_n|_{p^*}$ do not converge to $0$.
Therefore, passing to a subsequence and using Lemma \ref{lem:Solimini}, $T_{s_n,y_n}(v_n)\weakto v_0$ in $\cV$ for some $v_0\neq 0$, $(s_n)\subset (0,\infty)$ and $(y_n)\subset \R^3$. Taking subsequences again we also have that $T_{s_n,y_n}(v_n)\to v_0$ a.e. in $\R^3$ and in view of Theorem \ref{Th:Concentration}, $w(T_{s_n,y_n}(v_n))\weakto w(v_0)$ in $\cW$ and $w(T_{s_n,y_n}(v_n))\to w(v_0)$ a.e. in $\R^3$. We set $u_0:=v_0+w(v_0)\neq 0$ and we assume without loss of generality that $s_n=1$ and $y_n=0$. 

Take any $u\in \cM\setminus\{0\}$. Then $w(u)=0$
and hence
\begin{equation}\label{eq:ineqSobM}
\int_{\R^3}|\curlop u|^p\,dx\geq S_{p,\curl}\Big(\int_{\R^3}|u|^{p^*}\,dx\Big)^{\frac{p}{p^*}}
\end{equation}
according to \eqref{eq:scurl}.
Consider the map $\psi:\Wc\setminus\cW\to\R$ given by 
$$\psi(u):=\Big(\int_{\R^3}|u+w(u)|^{p^*}\,dx\Big)^{\frac{p}{p^*}}.$$
As $u_n\in\cN\subset\cm\setminus\{0\}$, we see using \eqref{eq:scurl} that
\begin{equation}\label{eqpsiun}
\psi(u_n)=\Big(\int_{\R^3}|u_n|^{p^*}\,dx\Big)^{\frac{p}{p^*}}=\Big(\int_{\R^3}|\curlop u_n|^{p}\,dx\Big)^{\frac{p}{p^*}}\to S_{p,\curl}^{\frac{p}{p^*-p}}.
\end{equation}
Let $z\in \Wc$.
Since $\psi(u_n)$ is bounded away from zero (and hence $u_n$ is bounded away from $\cW$), we can find $\delta>0$ such that
$$u_n+\lambda z\in \Wc\setminus\cW \hbox{ for all }\lambda\in [-\delta,\delta]\hbox{ and }n\geq 1.$$
Observe that by \eqref{eq:ineqSobM}
$$
\int_{\R^3}|\curlop (u_n+\lambda z)|^p\,dx =\int_{\R^3}|\curlop (u_n+\lambda z+w(u_n+\lambda z))|^p\,dx\geq S_{p,\curl}\, \psi(u_n+\lambda z)
$$
and, passing to a subsequence and using \eqref{eqpsiun}, we may assume
$$\Big|\int_{\R^3}|\curlop u_n|^p\,dx-S_{p,\curl}\,\psi(u_n)\Big|<\frac1n.$$
Hence by the mean value theorem there is $\theta_n=\theta_n(\lambda)\in [0,1]$ such that
\begin{eqnarray}\label{eq:firstspcurl}
&&\quad\int_{\R^3}|\curlop (u_n+\lambda z)|^p\,dx-\int_{\R^3}|\curlop u_n|^p\,dx \geq 
S_{p,\curl} \big(\psi(u_n+\lambda z)-\psi(u_n)\big)-\frac{1}{ n}\\ \nonumber
&& = S_{p,\curl}\,\psi'(u_n+\theta_n\lambda z)\lambda z-\frac1n = S_{p,\curl} \; p\Big(\int_{\R^3}|\wt u_n|^{p^*}\,dx\Big)^{-\frac{p^*-p}{p^*}}\int_{\R^3}|\wt u_n|^{p^*-2}\langle \wt u_n, \lambda z\rangle \,dx-\frac{1}{n}
\end{eqnarray}
where $\wt u_n:=u_n+\theta_n\lambda z+w(u_n+\theta_n\lambda z)$. To show that the derivative $\psi'$ is as claimed above, we put $I(u+w(u)) := \ir3|u+w(u)|^{p^*}\,dx$. Then $\psi(u) = I(u+w(u))^{p/p^*}$. Now we use the fact that the derivative of $I$ evaluated at $z$ is $I'(u+w(u))z$, see the proof of property (ii) on p. 4320 in \cite{BartschMederski2}. Although the setup there is somewhat different, an inspection of the argument shows that the conclusion holds true also in our situation. 

Next we prove two crucial facts which we formulate as claims. Recall that $u_n\to u_0$ a.e. in $\r3$.

\medskip

{\em Claim 1. $\curlop u_n\to\curlop u_0$ a.e. in $\R^3$, up to a subsequence.}\\
Similar results are known for gradients, see e.g. \cite{BoccardoMurat} or \cite{SzulkinWillem}.  In the case of $\curlop$, however, the argument is substantially more involved and requires a different approach.

By Theorem \ref{Th:Concentration}, 
\[|v_n-v_0|^{p^*}\rh \ol\rho := \sum_{x\in I}\rho_x\delta_x \quad \text{in } \cM(\r3). 
\]
Fix $-1/2<\lambda<0$ and take a compact set $K\subset \r3\setminus I$. Let $v_{n,m} := \vp(v_n-v_m)$ where $\vp\in\cC_0^\infty(\r3,[0,1])$ is such that $\vp=1$ in an open neighborhood of $K$. Since $v_n\rh v_0$, using the triangle inequality it is easy to see that
\[
\limsup_{n,m\to\infty}|v_{n,m}|_{p^*}\le 2\Big(\ir3\vp^{p^*}d\ol\rho\Big)^{1/p^*}.
\]
As $K\cap I=\emptyset$ and the above inequality holds for any $\vp$ as above, $|v_{n,m}|_{p^*, K}\to 0$.
By convexity and \eqref{eq:firstspcurl},
\begin{eqnarray*}
&&p\int_{\R^3}|\curlop (u_n+\lambda v_{n,m})|^{p-2}\langle \curlop (u_n +\lambda v_{n,m}), \lambda \curlop  v_{n,m}\rangle\,dx \\ 
&&\quad \ge \ir3|\curlop (v_n+\lambda v_{n,m})|^p\,dx - \ir3|\curlop v_n|^p\,dx \ge -C|\lambda| |v_{n,m}|_{p^*}-1/n.
\end{eqnarray*}
As $\lambda<0$, we obtain
\[
p\int_{\R^3}|\curlop (u_n+\lambda v_{n,m})|^{p-2}\langle \curlop (u_n + \lambda v_{n,m}), \curlop  v_{n,m}\rangle\,dx \le C |v_{n,m}|_{p^*}+1/(|\lambda|n).
\]
Interchanging $n$ and $m$ and using $v_{m,n}=-v_{n,m}$ gives
\[
p\int_{\R^3}|\curlop (u_m-\lambda v_{n,m})|^{p-2}\langle \curlop (u_m -\lambda v_{n,m}), \curlop  v_{n,m}\rangle\,dx \ge -C |v_{n,m}|_{p^*}-1/(|\lambda|m).
\]
To simplify notation we put $A_p(s) := |s|^{p-2}s$. Since $\curlop u_n = \curlop v_n$, the above two inequalities imply 
\begin{eqnarray}\label{Apineq}
\int_{\R^3}\langle A_p(\curlop (v_n +\lambda v_{n,m})) - A_p(\curlop (v_m -\lambda v_{n,m})), \curlop  v_{n,m}\rangle\,dx && \\
\le 2C|v_{n,m}|_{p^*}+\frac{1}{|\lambda|}\Big(\frac1n+\frac1m\Big). \nonumber &&
\end{eqnarray} 
We have
\begin{eqnarray*} 
\curlop v_{n,m} & = & \curlop(\vp(v_n-v_m)) = \vp\curlop(v_n-v_m) + \nabla\vp\times(v_n-v_m) \\\nonumber
&=& \vp\curlop(v_n-v_m+2\lambda v_{n,m}) -2\lambda\vp\curlop v_{n,m} + \nabla\vp\times(v_n-v_m),\nonumber
\end{eqnarray*}
and therefore
\[
(1+2\lambda\vp)\curlop v_{n,m} = \vp\curlop(v_n-v_m+2\lambda v_{n,m}) + \nabla\vp\times(v_n-v_m).
\]
Using this and putting
\begin{eqnarray*}
F_{n,m} := \big\langle A_p(\curlop (v_n +\lambda v_{n,m})) - A_p(\curlop (v_m -\lambda v_{n,m})),
\curlop (v_n-v_m + 2\lambda v_{n,m})\big\rangle,
\end{eqnarray*}
 \eqref{Apineq} can be rewritten as
\begin{eqnarray*}
&&\int_{\R^3}\frac{1}{1+2\lambda\vp}\langle A_p(\curlop (v_n +\lambda v_{n,m})) - A_p(\curlop (v_m -\lambda v_{n,m})), \nabla\vp\times(v_n-v_m)\rangle\,dx \\
&&\quad + \int_{\R^3}\frac{\vp}{1+2\lambda\vp}F_{n,m}\,dx \le  2C|v_{n,m}|_{p^*}+\frac{1}{|\lambda|}\Big(\frac1n+\frac1m\Big).
\end{eqnarray*}
Since $v_n\to v_0$ in $L^p_{loc}(\r3,\r3)$ and $1+2\lambda>0$, the first integral tends to 0 as $n,m\to\infty$. The inequality above holds for any $\vp$ which is 1 in some neighborhood of $K$. So recalling that $|v_{n,m}|_{p^*,K}\to 0$ and using $F_{n,m}\ge 0$ 
we obtain
\[
\int_K F_{n,m}\,dx \to 0.
\] 
By \cite{Simon},
\begin{eqnarray*}
	F_{n,m} & \gtrsim & |\curlop v_n-\curlop v_m+2\lambda \curlop v_{n,m}|^p \quad \text{if } p\ge 2, \\ 
	F_{n,m} & \gtrsim & (|\curlop (v_n+\lambda v_{n,m})|+|\curlop (v_m-\lambda v_{n,m})|)^{p-2} |\curlop v_n-\curlop v_m+ 2\lambda \curlop v_{n,m}|^2\\
&& \text{if } 1<p<2. 
\end{eqnarray*}
 If $p\ge2$, the first inequality and the fact that $v_{n,m}=v_n-v_m$ on $K$ give 
 \[
 \int_{K}|\curlop v_n-\curlop v_m|^p\,dx
 \lesssim\int_{K}F_{n,m}\,dx
 \longrightarrow0
 \]
 while if $1<p<2$, the second inequality and H\"older's inequality
 give
 \begin{eqnarray*}
 	&&\int_{K}|\curlop v_n-\curlop v_m|^p\,dx \lesssim\left(\int_{K}F_{n,m}\,dx\right)^{p/2}
 	\\
 	&&\qquad{}\times
 	\left(
 	\int_{K}
 	\big(|\curlop(v_n+\lambda v_{n,m})|
 	+|\curlop(v_m-\lambda v_{n,m})|\big)^p\,dx
 	\right)^{(2-p)/2}
 	\longrightarrow0.
 \end{eqnarray*}
 
Let $\Omega$ be a ball and $\delta>0$. Choose a compact set $K$ such that the measure $|\Omega\setminus K|<\delta$. For every $\varepsilon>0$,
 \begin{eqnarray*}
 	&&\big|\{x\in\Omega:
 	|\curlop v_n(x)-\curlop v_m(x)|>\varepsilon\}\big|\\
 	&&\quad\le
 	|\Omega\setminus K|
 	+\varepsilon^{-p}
 	\int_{K}|\curlop v_n-\curlop v_m|^p\,dx
\leq \delta +o(1).
 \end{eqnarray*}
Indeed, 
 \[
 \big|\{x\in\Omega\setminus K: |\curlop v_n(x)-\curlop v_m(x)|>\varepsilon\}\big| \le |\Omega\setminus K|
 \]
 and 
 \[
 \eps^p\big|\{x\in K:	|\curlop v_n(x)-\curlop v_m(x)|>\varepsilon\}\big|\ \le \int_{K}|\curlop v_n-\curlop v_m|^p\,dx.
 \]
Since $\delta$ is arbitrarily small,
$$\lim_{n,m\to\infty}\big|\{x\in\Omega:
|\curlop v_n(x)-\curlop v_m(x)|>\varepsilon\}\big|=0.$$
Therefore $(\curlop v_n)$ is Cauchy in measure on $\Omega$.
 After passing to a subsequence, it therefore converges
 a.e. in $\Omega$ to some $\wt v$.
 Since $v_n\rh v_0$ in $\Wc$,
the distributional limit of $\curlop v_n$ is
$\curlop v_0$. Thus $\wt v=\curlop v_0$ a.e. in $\Omega$,
and
\begin{equation}\label{eq15}
\curlop v_n\longrightarrow\curlop v_0
\quad\text{a.e. in }\Omega.
\end{equation}
Finally, applying this argument successively on
$\Omega_k:=B(0,k)$ and taking a diagonal subsequence,
we obtain $\curlop v_n\to\curlop v_0$ a.e. in $\r3$.
This completes the proof of {\em Claim 1}.

\medskip

{\em Claim 2. $u_0$ is a critical point of $J$.}\\
Recall $\wt u_n = u_n+\theta_n\lambda z+w(u_n+\theta_n\lambda z)$ and let $\lambda>0$ in \eqref{eq:firstspcurl}. By convexity again and by \eqref{eq:firstspcurl} we obtain
\begin{eqnarray*}
	&&p \int_{\R^3}|\curlop (u_n+\lambda z)|^{p-2}\langle \curlop (u_n+\lambda z),\lambda\curlop  z\rangle\,dx \\
	&& \qquad \geq \int_{\R^3}|\curlop (u_n+\lambda z)|^p\,dx-\int_{\R^3}|\curlop u_n|^p\,dx\\
	&&\qquad \geq S_{p,\curl} \ p\Big(\int_{\R^3}|\wt u_n|^{p^*}\,dx\Big)^{-\frac{p^*-p}{p^*}}\int_{\R^3}|\wt u_n|^{p^*-2}\langle \wt u_n, \lambda z\rangle \,dx-\frac{1}{n}.
\end{eqnarray*}
Passing to a subsequence, $\theta_n\to \theta_0=\theta_0(\lambda)\in [0,1]$. 
Since $u_n+\lambda z\rh u_0+\lambda z$, we obtain using Claim 1
\begin{eqnarray}
&&\int_{\R^3}|\curlop (u_0+\lambda z)|^{p-2}\langle \curlop (u_0+\lambda z), \curlop  z\rangle\,dx 
 \ge S_{p,\curl} \lim_{n\to\infty}\Big(\int_{\R^3}|\wt u_n|^{p^*}\,dx\Big)^{-\frac{p^*-p}{p^*}} \times \label{eq16} \\
&& \qquad \times \int_{\R^3}\langle |u_0+\theta_0\lambda z +w(u_0+\theta_0 \lambda z)|^{p^*-2}(u_0+\theta_0\lambda z +w(u_0+\theta_0 \lambda z), z\rangle \,dx. \nonumber
\end{eqnarray} 
Since
\begin{equation} \label{tight}
	\big||u_n+\theta_n\lambda z|^{p^*}-|u_n|^{p^*}\big| \lesssim  
	|\lambda z|\big(|u_n|^{p^*-1}+ |\lambda z|^{p^*-1}\big),
\end{equation}
the family $\big(|u_n+\theta_n\lambda z|^{p^*}-|u_n|^{p^*}\big)$  is uniformly integrable and tight.
Hence passing to a subsequence and using \eqref{eqeq}, Vitali's convergence theorem and  \eqref{eqpsiun},
\begin{eqnarray*}
 \lim_{n\to\infty}\int_{\R^3}|\wt u_n|^{p^*}\,dx &\leq& \lim_{n\to\infty}\int_{\R^3}|u_n+\theta_n\lambda z|^{p^*}\,dx \\ 
	&=& \lim_{n\to\infty} \Big(\ir3(|u_n+\theta_n\lambda z|^{p^*}-|u_n|^{p^*})\,dx +\ir3|u_n|^{p^*}\,dx \Big)  \\
	&=& \int_{\R^3}(|u_0+\theta_0\lambda z|^{p^*}-|u_0|^{p^*})\,dx + S_{p,\curl}^{\frac{p^*}{p^*-p}}.
\end{eqnarray*}
Since $w(u_n)=0$, the minimizing property of $w$ gives
\begin{eqnarray*}
	|u_n|_{p^*}
	&\le& |u_n+w(u_n+\theta_n\lambda z)|_{p^*}
	= |\wt u_n-\theta_n\lambda z|_{p^*}\\
	&\le& |\wt u_n|_{p^*}+\lambda|z|_{p^*}.
\end{eqnarray*}
Taking $\lambda>0$ sufficiently small and using
\eqref{eqpsiun} and the above inequality, we obtain
\begin{eqnarray*}
	\left(
	S_{p,\curl}^{\frac{1}{p^*-p}}-\lambda|z|_{p^*}
	\right)^{p^*}
	&\leq& \lim_{n\to\infty}\int_{\R^3}|\wt u_n|^{p^*}\,dx\leq
	\int_{\R^3}
	\big(|u_0+\theta_0\lambda z|^{p^*}-|u_0|^{p^*}\big)\,dx
	+S_{p,\curl}^{\frac{p^*}{p^*-p}}.
\end{eqnarray*}
Now, letting $\lambda\to 0$ in \eqref{eq16}, we get
\begin{eqnarray*}
	&&\int_{\R^3}|\curlop u_0|^{p-2}\langle \curlop u_0,\curlop  z\rangle\,dx
	\geq \int_{\R^3}|u_0+w(u_0)|^{p^*-2}\langle u_0+w(u_0),z\rangle \,dx.
\end{eqnarray*} 
Since $w(u_n)=0$ and $w(u_n)\rh w(u_0)$ according to Theorem \ref{Th:Concentration}, $w(u_0)=0$. Therefore $J'(u_0)z=0$ for any $z\in\Wc$ and $u_0$ is a solution to \eqref{eq} which proves \emph{Claim 2}.

\medskip

Now we can complete the proof. Since $u_0, u_n\in\cN$, 
we obtain using Fatou's lemma that
\begin{eqnarray*}
	\inf_\cN J & = &  J(u_n)+o(1) = J(u_n)-\frac1p J'(u_n)u_n + o(1) = \frac13|u_n|^{p^*}_{p^*} + o(1) \\
	& \ge & \frac13|u_0|^{p^*}_{p^*} + o(1) = J(u_0)-\frac1pJ'(u_0)u_0+o(1)=J(u_0) + o(1).
\end{eqnarray*}
Hence $J(u_0)=\inf_\cN J$. Moreover, $u_n\to u_0$ in $L^{p^*}(\r3,\r3)$, and using this it is easy to see that $u_n\to u_0$ in $\Wc$. Replacing $u_n$ with $T_{s_n,y_n}u_n$ we obtain $s_n^{3/p^*}u_n(s_n\cdot +y_n) \to u_0$ as claimed.

\medskip

(c) The first part of the claim follows from (b) and Lemma \ref{equiv}. It remains only to show that if $u_0$ satisfies equality in  \eqref{eq:neq}, then there exist unique $t(u_0)>0$ and $w\in\cW$ such that $t(u_0)(u_0+w)\in\cN$ and is a minimizer for $J|_\cN$. By \eqref{eqeq} and \eqref{eq:scurl}, $w=w(u_0)$. It is easy to see that $w(tu)=tw(u)$ for any $u$ and any $t> 0$ (the proof may be found in \cite[Lemma 4.2]{MedSz}). Since 
$$
J(t(u_0+w(u_0)) = \frac{t^p}p|\curlop u_0|_p^p -\frac{t^{p^*}}{p^*}|u_0+w(u_0)|^{p^*}_{p^*},
$$ 
there exists a unique $t=t(u_0)$ such that $\frac d{dt}J(t(u_0+w(u_0))=0$. Then $t(u_0)(u_0+w(u_0))\in\cN$ and
\[
S_{p,\curl} = \frac{|\nabla\times u_0|_p^p}{|u_0+w(u_0)|_{p^*}^p} = \frac{|\nabla\times t(u_0)u_0|_p^p}{|t(u_0)(u_0+w(u_0))|_{p^*}^p}.
\] 
By Lemma \ref{equiv}, $t(u_0)(u_0+w(u_0))$ is a minimizer for $J|_\cN$.

\medskip

(a) By Lemma \ref{lemmaScurlS}, $S_{p,\curl} \ge S_p\cdot \cH_p$ and by part (b), there exists $u=v+w(v)$ for which $S_{p,\curl}$ is attained. Suppose $S_{p,\curl}=S_p\cdot \cH_p$. As
\[
S_p\Big(\ir3|v+w(v)|^{p^*}\,dx\Big)^{p/p^*} \le S_p\Big(\ir3|v|^{p^*}\,dx\Big)^{p/p^*} \le \ir3|\nabla v|^p\,dx,
\]
 all inequalities  in \eqref{eq:5} with $\eps=0$, and hence also in \eqref{eq:SpcurlSpH}, become equalities. It follows in particular that $\left|\nabla|v|\right| = |\nabla v|$ a.e. and $|v|$ is of the form \eqref{eq:AT}. Since $|v|$ is continuous and positive, $v/|v|\in \cD^{1,p}_{loc}(\r3,\r3)$ and by Lemma \ref{equality}, $v=|v|e_0$ where $e_0$ is a constant vector. Since $\div(v)=0$, this is impossible. Hence $S_{p,\curl}>S_p\cdot \cH_p$.
\end{altproof}

\section{Symmetry}\label{sec:symmetry}

Let 
$$
g=\begin{pmatrix} \cos \alpha & -\sin  \alpha & 0\\
\sin  \alpha &  \cos \alpha & 0\\
0& 0 & 1
\end{pmatrix} \in\cO:=\SO(2)\times\{1\}\subset \SO(3)
$$ 
with $\alpha \in\R$ and let $u=(u_1,u_2,u_3):\R^3\to\R^3$. We define 
$$(g\star u)(x) := g\cdot u(g^{-1} x)
=\begin{pmatrix}  u_1 (g^{-1}x) \cos\alpha - u_2 (g^{-1}x) \sin\alpha\\
u_1 (g^{-1}x) \sin\alpha + u_2 (g^{-1}x) \cos\alpha\\
u_3 (g^{-1}x)\end{pmatrix}$$ 
for $x\in\R^3$.
It is straightforward to verify that
\[
|\curlop (g\star u)(x)| = |\curlop u (g^{-1}x)|, 
\qquad x \in \mathbb{R}^3,\; g \in \mathcal{O}.
\]
Using this it is easy to see that $\cO$ induces an isometric action on $\Wc$, the functional $J$ is invariant under this action and so are the subspaces $\cV,\cW$ (cf. \cite[Proposition 6.1]{BartschMederski2}). In particular, $\Wco = \cV_\cO \oplus \cW_\cO$.

Recall that $T_{s,y}(u):= s^{3/p^*}u(s\cdot +y)$.

\begin{Lem}\label{lem:Solimini2}
Suppose that $(v_n)\subset \Wco$ is bounded. Then $v_n\to 0$ in $L^{p^*}(\R^3,\R^3)$ if and only if $T_{s_n,y_n}(v_n)\weakto 0$ in $\cDp$ for all $(s_n)\subset (0,\infty)$ and $(y_n)\subset \{0\}\times\{0\}\times\R$.
\end{Lem}

\begin{proof}	
In view of Lemma~\ref{lem:Solimini} it is enough to show that the latter condition is sufficient. 
	
Suppose $v_n\not\to 0$ in $L^{p^*}(\R^3,\R^3)$. Then there exist $(s_n)\subset (0,\infty)$ and $(y_n)\subset\R^3$ such that $T_{s_n,y_n}(v_n)\weakto v_0\ne 0$ in $\cDp$. 
We can find $c>0$ and $n_0$ such that  
\begin{equation*}
\ir3\langle |\nabla v_0|^{p-2}\nabla v_0, \nabla T_{s_n,y_n}(v_n)\rangle\,dx \ge c, \quad n\ge n_0.
\end{equation*}
Let $B(x,r)$ denote the ball of radius $r$ and center at $x\in\r3$. If $R$ is large enough, then, putting $s_n^{-1}y_n =: z_n = (z_n^1,z_n^2,z_n^3)\in \R^3$, we have
\begin{eqnarray} \label{bddaway}
\qquad|\nabla T_{s_n,0}(v_n)|_{L^p(B(z_n,R),\r3)} & \gtrsim & \int_{B(z_n,R)}\langle |\nabla v_0(\cdot-z_n)|^{p-2}\nabla v_0(\cdot-z_n), \nabla T_{s_n,0}(v_n)\rangle\,dx \\
& = & \int_{B(0,R)}\langle |\nabla v_0|^{p-2}\nabla v_0, \nabla T_{s_n,y_n}(v_n)\rangle\,dx \ge c/2, \quad n\ge n_0. \nonumber
\end{eqnarray}
Let $n(R)$ be the maximal number of disjoint balls in the family $\{B(gz_n,R)\}_{g\in\cO}$. Then $n(R)\to\infty$ if $|(z_n^1,z_n^2)|\to\infty$. As
\[
\ir3|\nabla T_{s_n,y_n}(v_n)|^p\,dx = \ir3|\nabla T_{s_n,0}(v_n)|^p\,dx 
\ge n(R)\int_{B(z_n,R)}|\nabla T_{s_n,0}(v_n)|^p\,dx \ge n(R)c_0 
\]
for some $c_0>0$ and all $n\ge n_0$ according to  \eqref{bddaway}, $n(R)$ is bounded and hence so is the sequence $(z_n^1,z_n^2)$. Passing to a subsequence, $(z_n^1,z_n^2)\to (z_0^1,z_0^2).$

 Let $\Phi\in \cC_0^\infty(\r3,\R^{3\times 3})$. We have
\begin{eqnarray*}
&&\lim_{n\to\infty}\ir3\langle\Phi,\nabla T_{s_n,(0,0,y_n^3)}(v_n)\rangle\,dx
= \lim_{n\to\infty}\ir3\langle \Phi(\cdot+(z_n^1,z_n^2,0)),\nabla T_{s_n,y_n}(v_n)\rangle\,dx \\
&& \quad= \lim_{n\to\infty}\ir3\langle \Phi(\cdot+(z_0^1,z_0^2,0)),\nabla T_{s_n,y_n}(v_n)\rangle\,dx \\
&& \quad=\ir3\langle \Phi(\cdot+(z_0^1,z_0^2,0)),\nabla v_0\rangle\,dx 
= \ir3\langle \Phi,\nabla v_0(\cdot-(z_0^1,z_0^2,0))\rangle\,dx,
\end{eqnarray*}
i.e. $T_{s_n,(0,0,y_n^3)}(v_n)\weakto  v_0(\cdot-(z_0^1,z_0^2,0))\neq 0$.
\end{proof}

\begin{altproof}{Theorem \ref{Th:main2}} It is clear that $S_{p,\curl}^{\cO}\geq S_{p,\curl}$.
In order to prove (b) and (c), we proceed as in the proof of Theorem~\ref{Th:main1}, 
with the only difference that we now apply Lemma~\ref{lem:Solimini2} instead of Lemma \ref{lem:Solimini}. Moreover, we use the Palais principle of symmetric criticality \cite[Theorem 5.1]{Palais} to obtain solutions of \eqref{eq}.

We show (d) and $S_{3/2,\curl}^{\cO}\leq4\pi$.
Let $p=3/2$ and let $u$ be given by \eqref{defu}.
Then direct computations show that $\div(u)\neq 0$, $\div(|u|u)=0$, $\curlop u=4(1+|x|^2)^{-1}u$ and $|u|=3(1+|x|^2)^{-1}|w|$. Hence
$$|\curlop u|^{-1/2}\curlop u=12^{-1/2}(1+|x|^2)|w|^{-1/2}4(1+|x|^2)^{-1}u=\frac{2}{\sqrt{3}}|w|^{-1/2}u$$
and
$$\curlop\big(|\curlop u|^{-1/2}\curlop u\big)=\frac{8}{\sqrt{3}}|w|^{-1/2}(1+|x|^2)^{-1}u=\frac{8}{3\sqrt{3}}|w|^{-3/2}|u|u.$$
So $u$ given by \eqref{defu} solves \eqref{eq} provided that $|w|=4/3$. It is easy to see that this $u$ is $\cO$-equivariant if (and only if) $w=(0,0,\pm4/3)$.
Now observe that
\begin{equation*}
J(u)=J(u)-\frac23J'(u)(u)=\frac13\int_{\R^3}|u|^{3}\,dx=\frac{64}3\int_{\R^3}\frac{1}{(1+|x|^2)^3}\,dx 
=\frac{16}3\pi^2. 
\end{equation*}
Since $u\in \Wco\setminus\cW$ and $\inf_{\cN_\cO}J(u)=\frac13\big(S_{3/2,\curl}^{\cO}\big)^{2}$, we obtain that
$$S_{3/2,\curl}^{\cO}\leq 4\pi.$$
\end{altproof}

\begin{altproof}{Theorem \ref{Th:main3}}
Since $J$ is invariant with respect to $\cT$ and $\cS$, we proceed as in the proof of Theorems ~\ref{Th:main1} and \ref{Th:main2}. Note only that Lemma \ref{lem:Solimini2} applies also here because if $v_n$ are respectively as in \eqref{eq:formulauU} and \eqref{eq:formulauU2}, then so are $T_{s_n,y_n}(v_n)$ with $y_n=(0,0,y_n^3)$.  
\end{altproof}

\section{New approach to minimizing sequences in the Sobolev inequality}\label{sec:Lions}

The proof of Theorem \ref{Th:LionsW} is basically the same as that of Theorem \ref{Th:main1}(b) though some details become simpler. Instead of working in $\Wc$ we consider the space $\cD^{1,p}(\R^N)$ (note that there are no subspaces $\cV$ and $\cW$ here). For the reader's convenience and since the method appears to be applicable to other Sobolev-type inequalities, below we provide most of the details.

The energy functional associated with \eqref{eq:plap} is given by
\begin{equation}\label{eq:actionN}
J(u):=\frac1p\int_{\R^N} |\nabla u|^p\,dx- \frac1{p^*}\int_{\R^N} |u|^{p^*}\, dx,\quad u\in \cD^{1,p}(\R^N).
\end{equation}
Let
\begin{eqnarray}\label{def:NehN}
\cN &:= & \{u\in \cD^{1,p}(\R^N)\setminus\{0\}: J'(u)u=0\} \\
& = & \Big\{u\in \cD^{1,p}(\R^N)\setminus\{0\}: \int_{\R^N}|\nabla  u|^p=\int_{\R^N}|u|^{p^*}\, dx\Big\} \nonumber
\end{eqnarray}
be the usual  {\em Nehari manifold}.

\medskip

\begin{altproof}{Theorem \ref{Th:LionsW}}
Let  $(u_n)\subset \cD^{1,p}(\R^N)\setminus\{0\}$ be a minimizing sequence for \eqref{eq:Sobolevineq}. The mapping $t\mapsto J(tu_n)$, $t>0$, has a unique critical point $t_n$ and it is easy to see that  $tu_n\in\cN$ if and only if $t=t_n$. Since also the sequence $(t_nu_n)$ is minimizing, we may assume without loss of generality that $(u_n)\subset \cN$.  Since
	\begin{equation} \label{acN}
	J(u_n) = J(u_n) -\frac1{p^*}J'(u_n)u_n = \frac1N|\nabla u_n|^p_p= \frac1N|u_n|^{p^*}_{p^*} \quad \text{and } \quad  \frac{|\nabla u_n|_p^p}{|u_n|_{p^*}^p} \to S_p,
	\end{equation}
	 it follows that $(u_n)$ is bounded and bounded away from 0 both in $\cD^{1,p}(\R^N)$ and in $L^{p^*}(\rn)$, hence $|u_n|_{p^*}$ does not converge to $0$.
	Therefore, passing to a subsequence and using \cite{Solimini}, 
	$$T_{s_n,y_n}(u_n):=s_n^{N/p^*}u_n(s_n\cdot+y_n)\weakto u_0\hbox{ in } \cD^{1,p}(\R^N)$$ 
	for some $u_0\neq 0$, $(s_n)\subset (0,\infty)$ and $(y_n)\subset \R^N$. As $T_{s_n,y_n}$ is an isometric isomorphism of $\cD^{1,p}(\R^N)$, taking subsequences again we also have that $T_{s_n,y_n}(u_n)\to u_0$ a.e. in $\R^N$. We assume without loss of generality that $s_n=1$ and $y_n=0$. 
	
	Consider the map $\psi:\cD^{1,p}(\R^N)\setminus\{0\}\to\R$ given by 
	$$\psi(u):=\Big(\int_{\R^N}|u|^{p^*}\,dx\Big)^{\frac{p}{p^*}}$$
	and note that \eqref{acN} implies
	$\psi(u_n)\to S_p^{\frac{p}{p^*-p}}$.

	Let $z\in \cD^{1,p}(\R^N)$.
	Since $\psi(u_n)$ is bounded away from zero, we can find $\delta>0$ such that
	$$u_n+\lambda z\in\cD^{1,p}(\R^N)\setminus\{0\} \hbox{ for all }\lambda\in [-\delta,\delta]\hbox{ and }n\geq 1.$$
	Observe that by \eqref{eq:Sobolevineq}
	$$
	\int_{\R^N}|\nabla (u_n+\lambda z)|^p\,dx\geq S_p\, \psi(u_n+\lambda z),
	$$
	and passing to a subsequence we may assume
	$$\Big|\int_{\R^N}|\nabla u_n|^p\,dx-S_p\,\psi(u_n)\Big|<\frac1n.$$
	Hence by the mean value theorem there is $\theta_n=\theta_n(\lambda)\in [0,1]$ such that
	\begin{eqnarray}\label{eq:firstspcurlN}
	&&\quad\int_{\R^N}|\nabla(u_n+\lambda z)|^p\,dx-\int_{\R^N}|\nabla u_n|^p\,dx \geq 
	S_p \big(\psi(u_n+\lambda z)-\psi(u_n)\big)-\frac{1}{ n}\\ \nonumber
	&& = S_{p}\,\psi'(u_n+\theta_n\lambda z)\lambda z-\frac1n = S_p \, p\Big(\int_{\R^N}|\wt u_n|^{p^*}\,dx\Big)^{-\frac{p^*-p}{p^*}}\int_{\R^N}|\wt u_n|^{p^*-2}\langle \wt u_n, \lambda z\rangle \,dx-\frac{1}{n}
	\end{eqnarray}
	where $\wt u_n:=u_n+\theta_n\lambda z$. 
	
	Next we prove two similar claims as in the proof of Theorem \ref{Th:main1}. Recall that $u_n\to u_0$ a.e. in $\R^N$.
	
	\medskip
	
	{\em Claim 1. $\nabla u_n\to\nabla u_0$ a.e. in $\R^N$, up to a subsequence.}

We repeat the argument of Claim 1 in the proof of
Theorem \ref{Th:main1}, replacing $v_n$ by $u_n$,
the curl by the gradient, and $\R^3$ by $\R^N$.
Here we use the scalar concentration-compactness lemma,
which ensures that a weak-star limit of
$|u_n-u_0|^{p^*}\,dx$ is purely atomic.
The same localization and monotonicity argument yields
the claim.
		
	{\em Claim 2. $u_0$ is a critical point of $J$.}\\
	Recall $\wt u_n = u_n+\theta_n\lambda z$ and let $\lambda>0$ in \eqref{eq:firstspcurlN}. By convexity again and by \eqref{eq:firstspcurlN} we obtain
	\begin{eqnarray*}
		&&p \int_{\R^N}|\nabla (u_n+\lambda z)|^{p-2}\langle \nabla (u_n+\lambda z),\lambda\nabla  z\rangle\,dx \\
		&& \qquad \geq \int_{\R^N}|\nabla (u_n+\lambda z)|^p\,dx-\int_{\R^N}|\nabla u_n|^p\,dx\\
		&&\qquad \geq S_{p} \, p\Big(\int_{\R^N}|\wt u_n|^{p^*}\,dx\Big)^{-\frac{p^*-p}{p^*}}\int_{\R^N}|\wt u_n|^{p^*-2}\langle \wt u_n, \lambda z\rangle \,dx-\frac{1}{n}.
	\end{eqnarray*}
	Passing to a subsequence, $\theta_n\to \theta_0=\theta_0(\lambda)\in [0,1]$. 
	Since $u_n+\lambda z\rh u_0+\lambda z$, we obtain
	\begin{eqnarray}
	&&\int_{\R^N}|\nabla (u_0+\lambda z)|^{p-2}\langle \nabla (u_0+\lambda z), \nabla  z\rangle\,dx 
	\ge S_{p} \lim_{n\to\infty}\Big(\int_{\R^N}|\wt u_n|^{p^*}\,dx\Big)^{-\frac{p^*-p}{p^*}} \label{eq16N} \times \\
	&& \qquad \times \int_{\R^N}\langle |u_0+\theta_0\lambda z|^{p^*-2}(u_0+\theta_0\lambda z, z\rangle \,dx. \nonumber
	\end{eqnarray} 
	Since
	the family $\big(|u_n+\theta_n\lambda z|^{p^*}-|u_n|^{p^*}\big)$  is uniformly integrable and tight (see \eqref{tight}),
passing to a subsequence and using Vitali's convergence theorem we get
	\begin{eqnarray*}
	\left(
S_{p}^{\frac{1}{p^*-p}}-\lambda|z|_{p^*}
\right)^{p^*}&\leq&		\lim_{n\to\infty}\int_{\R^N}|\wt u_n|^{p^*}\,dx \leq \lim_{n\to\infty}\int_{\R^N}|u_n+\theta_n\lambda z|^{p^*}\,dx \\ 
		&=& \lim_{n\to\infty} \Big(\int_{\R^N}(|u_n+\theta_n\lambda z|^{p^*}-|u_n|^{p^*})\,dx +\int_{\R^N}|u_n|^{p^*}\,dx \Big)  \\
		&=& \int_{\R^N}(|u_0+\theta_0\lambda z|^{p^*}-|u_0|^{p^*})\,dx + S_{p}^{\frac{p^*}{p^*-p}}.
	\end{eqnarray*}
	Hence using \eqref{eq16N},  
	and letting $\lambda\to 0$, 
	\begin{eqnarray*}
		&&\int_{\R^N}|\nabla u_0|^{p-2}\langle \nabla u_0,\nabla  z\rangle\,dx
		\geq \int_{\R^N}|u_0|^{p^*-2}\langle u_0,z\rangle \,dx.
	\end{eqnarray*}
	Thus $J'(u_0)z=0$ and since $z$ is arbitrary, $u_0$ is a solution to \eqref{eq:plap} which proves \emph{Claim 2}.

	Finally, since $u_0, u_n\in\cN$, 
	we obtain using Fatou's lemma that
	\begin{eqnarray*}
		\inf_\cN J & = &  J(u_n)+o(1) = J(u_n)-\frac1p J'(u_n)u_n + o(1) = \frac1N|u_n|^{p^*}_{p^*} + o(1) \\
		& \ge & \frac1N|u_0|^{p^*}_{p^*} + o(1) = J(u_0)-\frac1pJ'(u_0)u_0+o(1)=J(u_0) + o(1).
	\end{eqnarray*}
	So $J(u_0)=\inf_\cN J$, hence  $u_n\to u_0$ in $L^{p^*}(\R^N)$ and $u_n\to u_0$ in $\cD^{1,p}(\R^N)$. Replacing $u_n$ with $T_{s_n,y_n}u_n$ we obtain $s_n^{N/p^*}u_n(s_n\cdot +y_n) \to u_0$.
\end{altproof}

\appendix

\section{Density lemma}\label{app:densisty}

\begin{Lem} \label{density}
$\Wc$ is the closure of $\cC^{\infty}_0(\R^3,\R^3)$ with respect to $\|\cdot\|$.
\end{Lem}

\begin{proof}
Let $\chi_R\in \cC_0^\infty(\R^3)$ be such that $|\nabla\chi_R|\le 2/R$, $\chi_R=1$ for $|x|\le R$ and $\chi_R= 0$ for $|x|\ge 2R$. Take $u=(u_1,u_2,u_3)\in \Wc$. Then $\chi_Ru\to u$ in $L^{p^*}(\R^3,\R^3)$ as $R\to\infty$. We have 
	\begin{equation} \label{chi}
		\partial_i(\chi_Ru_j)-\partial_j(\chi_Ru_i) = (\partial_i\chi_R)u_j - (\partial_j\chi_R)u_i + \chi_R(\partial_iu_j-\partial_ju_i),
		\quad i\ne j
	\end{equation}
	and
	\[
	\ir3|\partial_i\chi_R|^p|u_j|^p\,dx \le \left(\int_{R\le|x|\le 2R}|\partial_i\chi_R|^3\,dx\right)^{p/3}\left(\int_{R\le|x|\le 2R}|u_j|^{p^*}
	\,dx\right)^{p/p^*}.
	\]
	Since
	\[
	\int_{R\le|x|\le 2R}|\partial_i\chi_R|^3\,dx \le C
	\]
	for some $C>0$,
$(\partial_i\chi_R)u_j\to 0$ in $L^p(\r3)$. As $\partial_iu_j-\partial_ju_i\in L^p(\R^3)$, it follows that the left-hand side in \eqref{chi} tends to $\partial_iu_j-\partial_ju_i$ in $L^p(\R^3)$ as $R\to\infty$. 
	Hence $\chi_R u \to u$ in $\Wc$ and functions of compact support are dense in $\Wc$.
	
	Suppose now  $u\in \Wc$ has compact support. Clearly, $j_\eps*u\to u$ in $L^{p^*}(\R^3,\R^3)$ as $\eps\to 0$ where $j_\eps$ is the standard mollifier. Since
	\begin{equation*}
		\partial_i(j_\eps*u_j)-\partial_j(j_\eps*u_i) = j_\eps*(\partial_iu_j-\partial_ju_i)
	\end{equation*}
	and $\partial_iu_j-\partial_ju_i\in L^p(\R^3)$, the right-hand side above tends to $\partial_iu_j-\partial_ju_i$ in $L^p(\R^3)$ as $\eps\to 0$. This completes the proof.
\end{proof}

\section{Concentraction-compactness result}\label{appendix:concom}

\begin{altproof}{Theorem \ref{Th:Concentration}}
	Let $\vp\in\cC_0^{\infty}(\R^3)$. Similarly as in the proof of  \cite[Concentration-Compactness Lemma II]{Struwe} (see p. 45 there)
	we obtain
	\begin{equation*} \label{eq:standardLions2}
	\Big(\int_{\R^3}|\vp|^{p^*}d\bar\rho\Big)^{1/p^*} \le S_p^{-1/p}\Big(\int_{\R^3}|\vp|^p\,d\bar{\mu}\Big)^{1/p}
	\end{equation*}
	where 
	$\bar{\mu}:=\mu-|\nabla v_0|^p$ and $\bar{\rho}:=\rho-|v_0|^{p^*}$ (cf. also \cite[(3.4) and (3.5)]{MedSz}).
	Set $I=\{x\in\R^3: \mu_x:=\mu(\{x\})>0\}$. Since $\mu$ is finite and $\mu, \bar\mu$ have the same singular set,  $I$ is at most countable,
	$\mu\geq |\nabla v_0|^p+\sum_{x\in I}\mu_x\delta_x$
	and by the same argument as in \cite[p. 46]{Struwe},
	$\bar\rho = \sum_{x\in I}\rho_x\delta_x$. See also \cite[Proposition 4.2]{wa}.
	
	Using \eqref{eq:ineqF}, 
	$$\int_{\R^3}|v_n+w(v_n)|^{p^*}\,dx \le \int_{\R^3}|v_n|^{p^*}$$
	and	we see that $(w(v_n))$ is bounded.
	Thus, taking a subsequence we may assume $w(v_n)\weakto w_0$ in $\cW$ for some $w_0\in\cW$. 
	
	We shall show that 
	$w(v_n)\to w_0$ a.e. in $\R^3$ arguing as in the proof of \cite[Theorem 3.1]{MedSz}.
	Fix $l\geq 1$.
	In view of \cite[Lemma 1.1]{Le} there exists $\xi_n\in W^{1,p^*}(B(0,l))$ such that $w(v_n)=\nabla\xi_n$ and we may assume $\int_{B(0,l)}\xi_n\,dx = 0$. So by the Poincar\'e inequality, $(\xi_n)$ is bounded in $W^{1,p^*}(B(0,l))$
	and passing to a subsequence, $\xi_n\rh \xi$ in $W^{1,p^*}(B(0,l))$ and $\xi_n\to \xi$ in $L^{p^*}(B(0,l))$ for some $\xi\in W^{1,p^*}(B(0,l))$. 
	Let $\vp\in\cC_0^{\infty}(B(0,l))$.
	Since $\nabla (|\vp|^{p^*}(\xi_n-\xi)) = \nabla(|\vp|^{p^*})(\xi_n-\xi) + |\vp|^{p^*}(w(v_n)-\nabla\xi) \in\cW$, it follows from \eqref{eq:eqf} that
	\begin{eqnarray*}
	&& \int_{\R^3}|\vp|^{p^*}\langle |v_n+w(v_n)|^{p^*-2}(v_n+w(v_n)),w(v_n)-\nabla\xi\rangle \,dx \\  
	&& \qquad + \int_{\R^3}\langle |v_n+w(v_n)|^{p^*-2}(v_n+w(v_n)),\nabla (|\vp|^{p^*})(\xi_n-\xi)\rangle\,dx = 0
	\end{eqnarray*}
	which gives
	\begin{eqnarray*}
		&&\int_{\R^3}|\vp|^{p^*}\langle |v_n+w(v_n)|^{p^*-2}(v_n+w(v_n)),w(v_n)-\nabla\xi\rangle \,dx\\ 
		&&=-\int_{\R^3}\langle |v_n+w(v_n)|^{p^*-2}(v_n+w(v_n)),\nabla (|\vp|^{p^*})(\xi_n-\xi)\rangle\,dx=o(1)
	\end{eqnarray*}
	as $n\to\infty$. Since $w(v_n)\rh \nabla\xi$ in $L^{p^*}(B(0,l))$,  
	$$
	\int_{\R^3}|\vp|^{p^*}\langle |v_0+\nabla\xi|^{p^*-2}(v_0+\nabla\xi),w(v_n)-\nabla\xi\rangle \,dx=o(1),
	$$
	hence we obtain
	\begin{equation} \label{eq:est1}
	\int_{\R^3}|\vp|^{p^*}\langle |v_n+w(v_n)|^{p^*-2}(v_n+w(v_n))-|v_0+\nabla\xi|^{p^*-2}(v_0+\nabla\xi),\,  w(v_n)-\nabla\xi\rangle \,dx =o(1).
	\end{equation}
	For any $k\geq 1$, $|u_1-u_2|\ge \frac1k$, $|u_1|, |u_2| \le k$ and a constant $c_0>0$ we have (see e.g. \cite{Simon})
	\begin{equation} \label{mk}
	0 < m_{k} := \frac{c_0}{k^{p^*}} \le c_0|u_1-u_2|^{p^*} \le \langle |u_1|^{p^*-2}u_1-|u_2|^{p^*-2}u_2,\, u_1-u_2\rangle
	\end{equation}
	if $p^*\ge 2$ (i.e. $p\ge 6/5$) and
	\begin{equation} \label{mk2}
	0 < m_{k} :=\frac{4c_0'}{(2k)^{4-p^*}} \le c_0'(|u_1|+|u_2|)^{p^*-2}|u_1-u_2|^2 \le \langle |u_1|^{p^*-2}u_1-|u_2|^{p^*-2}u_2,\, u_1-u_2\rangle
	\end{equation}
	if $p^*<2$.
	Let 
	$$
	\Omega_{n,k}:=\Big\{x\in \R^3: |v_n+w(v_n)-v_0-\nabla\xi|\geq  \frac1k, \ |v_n+w(v_n)|\le k,\ |v_0+\nabla\xi|\le k \Big\}
	$$
	and set $u_1=v_n+w(v_n)$, $u_2 = v_0+\nabla\xi$ in respectively \eqref{mk} and \eqref{mk2}. Now  taking into account \eqref{eq:est1} and H\"older's inequality, we get for each fixed $k$
	\begin{eqnarray*}
		&&m_k\int_{\Omega_{n,k}}|\vp|^{p^*}\,dx \\ 
		&& \quad \le \int_{\R^3}|\vp|^{p^*}\langle |v_n+w(v_n)|^{p^*-2}(v_n+w(v_n))-|v_0+\nabla\xi|^{p^*-2}(v_0+\nabla\xi), \, v_n-v_0\rangle \,dx +o(1)\\
		&& \quad \lesssim \Big(\int_{\R^3}|\vp|^{p^*}|v_n-v_0|^{p^*}\,dx\Big)^{1/p^*} +o(1) = 
		 \Big(\int_{\R^3}|\vp|^{p^*}\,d\bar{\rho}\Big)^{1/p^*}+o(1).
	\end{eqnarray*} 
	Since $\vp\in\cC_0^{\infty}(B(0,l))$ is arbitrary, 
	\begin{equation}\label{eq:Borel}
	m_k|\Omega_{n,k}\cap E|\lesssim \big(\bar{\rho}(E)\big)^{1/p^*}+o(1)
	\end{equation}
	for any Borel set $E\subset B(0,l)$ (here $|\cdot|$ denotes the Lebesgue measure). We find a decreasing sequence of open sets $E_k\supset I$ such that $|E_k|<1/2^{k+1}$. Then, taking $E=B(0,l)\setminus E_k$ in \eqref{eq:Borel}, we have
	$m_k|\Omega_{n,k}\cap (B(0,l)\setminus E_k)|=o(1)$ as $n\to\infty$ because $\supp(\bar\rho)\subset I$. Hence 
	we can find a sufficiently large $n_k$ such that $|\Omega_{n_k,k}\cap B(0,l)|<1/2^k$ and we obtain
	$$\Big|\bigcap_{j=1}^{\infty}\bigcup_{k=j}^{\infty}\Omega_{n_k,k}\cap B(0,l)\Big| \le \lim_{j\to\infty}\sum_{k=j}^{\infty}|\Omega_{n_k,k}\cap B(0,l)|\leq \lim_{j\to\infty}\frac{1}{2^{j-1}}=0.$$
	If $x\notin \bigcap_{j=1}^{\infty}\bigcup_{k=j}^{\infty}\Omega_{n_k,k}$ and $x\in B(0,l)$, then
	\begin{eqnarray*}
		&&|v_{n_k}(x)+w(v_{n_k})(x)-v_0(x)-\nabla\xi(x)|< \frac1k,\hbox{ or }|v_{n_k}(x)+ w(v_{n_k})(x)|>k,\\
		&&\hbox{ or }|v_0(x)+\nabla\xi(x)|> k
	\end{eqnarray*}
	for all sufficiently large $k$. Since $v_{n_k}+w(v_{n_k})$ is bounded in $L^{p^*}(\Om,\R^3)$, the second and the third inequality above cannot hold on a set of positive measure for all large $k$. 
	So $v_{n_k}+w(v_{n_k})\to v_0+\nabla\xi$ a.e. in $B(0,l)$ and in particular, $w(v_{n_k})\to \nabla\xi$  a.e. in $B(0,l)$. Since $w(v_n)\weakto w_0$, $w_0=\nabla \xi$ a.e. in $B(0,l)$. Now employing the diagonal procedure, we find a subsequence of $w(v_n)$ which converges to $w_0$ a.e. in $\R^3=\bigcup_{l=1}^{\infty}B(0,l)$. 
	
	Let $q\in [1,p^*)$. For $\Omega\subset \R^3$ such that $|\Omega|<+\infty$, by H\"older's inequality we get
	\begin{equation*}
	\int_{\Om} |v_n-v_0+w(v_n)-w_0|^q\,dx
	\leq C |\Omega|^{1-\frac{q}{p^*}},
	\end{equation*}
	hence by the Vitali convergence theorem, $v_n-v_0+w(v_n)-w_0\to 0$ in $L^q_{loc}(\R^3)$ after passing to a subsequence.
	
	Finally, by  the Vitali convergence theorem again, 
	$$0 = \int_{\R^3}\langle |v_n+w(v_n)|^{p^*-2}(v_n+w(v_n)), w\rangle\,dx\to \int_{\R^3}\langle |v_0+ w_0|^{p^*-2}(v_0+w_0),  w\rangle\,dx$$
	for any $w\in\cW$, hence taking into account \eqref{eq:eqf} we get $w_0=w(v_0)$ which completes the proof.
\end{altproof}

\section{Relation between $v$ and $|v|$} \label{sec:relation}

Suppose $v\in\cDp$. Since $|v(x)|\le \sqrt{3}\max\{|v_i(x)|:i=1,2,3\}$,  it follows from \cite[Theorem 6.17 and Corollary 6.18]{LiebLossBook} that $|v|\in \cD^{1,p}(\r3)$. 

We show that \eqref{CauchySchwartz} holds. For notational convenience we denote the inner product in $\r3$ by $\cdot$ instead of $\langle\cdot,\cdot\rangle$.  We have 
\begin{equation} \label{cs}
(\partial_i|v|)^2 = \frac{(v\cdot \partial_iv)^2}{|v|^2} \le |\partial_iv|^2, \qquad i=1,2,3,
\end{equation}
where the middle term should be understood as $0$ if $v(x)=0$. Summing over $i$ gives the result. 

\begin{Lem} \label{equality}
Suppose that $v(x)\neq 0$ for all $x\in\R^3$ and $v/|v|\in \cD_{loc}^{1,p}(\r3,\r3)$. Then $\left|\nabla|v|\right| = |\nabla v|$ a.e. if and only if there exists a constant vector $e_0$ such that $v=|v|e_0$ a.e.
\end{Lem}

\begin{proof}
If $\left|\nabla|v(x)|\right| = |\nabla v(x)|$, then equality holds in \eqref{equality} for this $x$ and hence $\partial_ iv(x)=\lambda_i(x)v(x)$, $i=1,2,3$. Write $v(x) := |v(x)|e_0(x)$. Then
\[
\lambda_iv = \partial_iv = (\partial_i|v|)e_0 + |v|\partial_ie_0 = \frac{v\cdot\partial_iv}{|v|}\,e_0 + |v|\partial_i e_0 = \lambda_i|v|e_0 + |v|\partial_ie_0 = \lambda_iv + |v|\partial_ie_0.
\]
So $\nabla e_0 = 0$ a.e. and since $e_0\in\cD_{loc}^{1,p}(\r3,\r3)$, it is a constant vector as claimed.
\end{proof}

\section*{Statements and Declarations}
\noindent{\textbf{Funding.} J. Mederski was partly supported by the National Science Centre, Poland (Grant No. 2023/51/B/ST1/00968) and by the Thematic Research Programme
	"Variational and geometrical methods in partial differential equations",
	University of Warsaw, Excellence Initiative - Research University. A. Szulkin
	was partly supported by a grant from the Magnuson foundation of the Royal Swedish Academy of Sciences.}

\medskip
\noindent{\textbf{Competing interests.}}
The authors have no relevant financial or non-financial interests
	to disclose.

\medskip
\noindent{\textbf{Data availability.}
	No datasets were generated or analyzed in this theoretical study.

\end{document}